\documentclass[12pt,twoside]{amsart}
\usepackage{amssymb}
\usepackage{verbatim}
\usepackage{amsmath}
\usepackage{bm}
\usepackage{a4wide}
\usepackage[latin1]{inputenc}
\usepackage[T1]{fontenc}
\usepackage{times}
\usepackage{amssymb,latexsym}
\usepackage{enumerate}

\makeatletter
\newcommand{\sumprime}{\if@display\sideset{}{'}\sum%
            \else\sum'\fi}
\makeatother

\begin{document}

\numberwithin{equation}{section}

\newtheorem{theorem}{Theorem}[section]
\newtheorem{proposition}[theorem]{Proposition}
\newtheorem{conjecture}[theorem]{Conjecture}
\def\theconjecture{\unskip}
\newtheorem{corollary}[theorem]{Corollary}
\newtheorem{lemma}[theorem]{Lemma}
\newtheorem{observation}[theorem]{Observation}
\newtheorem{definition}{Definition}
\numberwithin{definition}{section} 
\newtheorem{remark}{Remark}
\def\theremark{\unskip}
\newtheorem{kl}{Key Lemma}
\def\thekl{\unskip}
\newtheorem{question}{Question}
\def\thequestion{\unskip}
\newtheorem{example}{Example}
\def\theexample{\unskip}
\newtheorem{problem}{Problem}

\thanks{Supported by National Natural Science Foundation of China, No. 11771089}

\address{School of Mathematical Sciences, Fudan University, Shanghai, 20043, China}

\email{boychen@fudan.edu.cn}

\title[Every bounded pseudoconvex domain with H\"older boundary is hyperconvex]{Every bounded pseudoconvex domain with H\"older boundary  is hyperconvex}

 \author{Bo-Yong Chen}
\date{}

\begin{abstract}
We show that every bounded pseudoconvex domain with H\"older boundary in $\mathbb C^n$ is hyperconvex.
\end{abstract}

\maketitle

\section{Introduction}

A bounded domain $\Omega$ in $\mathbb C^n$ is called hyperconvex \cite{Stehle} if there exists a negative continuous plurisubharmonic (psh) function $\rho$ on $\Omega$ such that $\{\rho<-c\}\subset\subset \Omega$ for all $c>0$. Actually, even the continuity condition of $\rho$ is superfluous (see \cite{Blocki}, Theorem 1.4.6). It is easy to see that hyperconvexity implies pseudoconvexity, while  there are pseudoconvex domains (e.g.,  the Hartogs triangle) which are not hyperconvex. There is  a close link between  the Bergman kernel (or the Bergman metric) and hyperconvex domains (see \cite{JP} and the references therein).  

On the positive side, it was Diederich-Fornaess \cite{DF77} who first showed that every  bounded pseudoconvex domain with $C^2$ boundary in $\mathbb C^n$ is hyperconvex. Actually, they proved that there  exists  a negative smooth psh function which is comparable to  $-\delta^\alpha$ for some $\alpha>0$, where $\delta$ is the boundary distance. 
 Kerzman-Rosay \cite{KerzmanRosay} showed that every   bounded pseudoconvex domain with $C^1$ boundary in $\mathbb C^n$  is hyperconvex.
  Demailly \cite{Demailly87} proved that every bounded pseudoconvex domain with Lipschitz boundary in $\mathbb C^n$ admits a negative smooth psh function which is comparable to $ (\log \delta)^{-1}$. The result of Diederich-Fornaess was extended by Harrington \cite{HarringtonPSH} to bounded pseudoconvex domains with Lipschitz boundaries in $\mathbb C^n$ and by Ohsawa-Sibony \cite{OS} to pseudoconvex domains with $C^2$  boundaries in $\mathbb P^n$. The latter was then generalized  by Harrington \cite{Harrington_2} to Lipschitz pseudoconvex domains in $\mathbb P^n$. Only recently, Avelin et al. \cite{LogLip} were able to show that the Lipschitz condition in Demailly's result can be relaxed to Log-Lipschitz for hyperconvexity.
 
For a bounded domain $\Omega\subset \mathbb C^n$ we say that $\partial \Omega$ is H\"older if it can be written locally as the graph of a H\"older continuous function.  The goal of this note is to show the following
 
 \begin{theorem}\label{th:Main}
 Every bounded pseudoconvex domain with H\"older boundary in\/ $\mathbb C^n$ is hyperconvex.
\end{theorem}

\begin{remark}
{\rm The conclusion fails if $\mathbb C^n$ is replaced by  $\mathbb P^n$ (cf. \cite{DO98}\footnote{The author is indebted to Professor Takeo Ohsawa for providing this reference.}, Theorem 1.4 and its proof).}
\end{remark}

The proof of Theorem \ref{th:Main} depends on an estimate for the relative extremal function (see Lemma \ref{lm:Key}), which is motivated by Gr\"uter-Widman's estimate for the capacity potential of sets in $\mathbb R^N$ (cf. \cite{GW}). The same method  also gives a quick proof of a weak version of Harrington's result that every bounded pseudoconvex domain with Lipschitz boundary has a positive hyperconvexity index (see \cite{ChenH-index} for the definition and applications). 

Actually, Theorem \ref{th:Main} can be extended to certain domains beyond H\"older boundary regularity (see \S\,5). Nevertheless, the following question still remains open. 

\begin{problem}
Is a bounded pseudoconvex domain  $\Omega\subset \mathbb C^n$ hyperconvex if $\partial \Omega$ can be written locally as the graph of a continuous function? 
\end{problem}

\section{A crucial estimate for the relative extremal function}
For a bounded domain $\Omega\subset{\mathbb C}^n$ we define  the\/ {\it relative extremal function\/} of a (fixed) closed ball $\overline{B}\subset \Omega$ by
 $$
 \varrho(z):=\varrho_{\overline{B}}(z):=\sup\left\{u(z):u\in PSH^-(\Omega),\,u|_{\overline{B}}\le -1\right\},
 $$
 where $PSH^-(\Omega)$ denotes the set of negative psh functions on $\Omega$.  Clearly, the upper semicontinuous regularization $\varrho^\ast$ of $\varrho$ is psh on $\Omega$.  Actually, we always have $\varrho=\varrho^\ast$. Indeed, for every element
$u$ in the family we have $u\le  h$ on $B'\backslash B$ in view of the maximum principle, where $h$ is  a harmonic function
equal to $-1$ on  $\partial B$ and equal to $0$ on the boundary of some bigger
ball $B'$. It follows that $\varrho\le h$ on $B'\backslash B$, which implies that $\varrho^\ast\le h$ on $B'\backslash B$ and $\sup_{\overline{B}}\varrho^\ast=\sup_{\partial B}\varrho^\ast\le -1$. Thus we have $\varrho\ge \varrho^\ast$. The opposite inequality is trivial\footnote{The author is indebted to one of the referees for providing this elementary argument.}.

 We have the following very useful estimate for $\varrho$. 

    \begin{lemma}\label{lm:Key}
   Let $\Omega$ be a bounded pseudoconvex domain in $\mathbb C^n$. Set $\Omega_t=\{z\in \Omega:\delta(z)>t\}$ for $t>0$.  Suppose  there exist a number $0<\alpha<1$ and  a family $\{\psi_t\}_{0<t\le t_0}$ of psh functions on $\Omega$ satisfying 
     $$
\sup_\Omega \psi_t<1\ \ \ \text{and}\ \ \ \inf_{\Omega\backslash \Omega_{\alpha t}} \psi_{t} > \sup_{\partial \Omega_{ t}} \psi_{t}
   $$ 
   for all $t$.
     Then we have 
    \begin{equation}\label{eq:H_0}
    \sup_{\Omega\backslash \Omega_r} (-\varrho) \le \exp\left[-\frac{1}{\log 1/\alpha} \int_{r/\alpha}^{ r_0} \frac{\kappa_\alpha(t)} t dt  \right]
    \end{equation}
    where $r/\alpha<r_0\ll1$ and
    $$
  \kappa_\alpha(t):= \frac{\inf_{\Omega\backslash \Omega_{\alpha t}}\psi_t-
 \sup_{\partial \Omega_t} \psi_t}{1-\sup_{\partial \Omega_t} \psi_t}.
   $$  
         \end{lemma}
    
    \begin{proof}
    We first assume that $\Omega$ is hyperconvex. Then $-1\le \varrho<0$ is a continuous negative psh function on $\Omega$ satisfying $\lim_{z\rightarrow \zeta} \varrho(z)=0$ for all $\zeta\in \partial \Omega$ (cf. \cite{Zahariuta}\footnote{The author is indebted to Professor Peter Pflug for providing this reference.}; see also  \cite{Blocki}, Proposition 3.1.3/vii)). Note that for $z\in \partial \Omega_t$ 
  \begin{equation}\label{eq:H_1}
  (1-\psi_{t}(z)) \sup_{\Omega\backslash  \Omega_t} (-\varrho) \ge \left[1-\sup_{\partial \Omega_t} \psi_t\right](-\varrho(z)).
  \end{equation}
  Since $\sup_\Omega \psi_t<1$ and $\lim_{z\rightarrow \zeta}\varrho(z)=0$ for  all $\zeta\in \partial \Omega$, we conclude that (\ref{eq:H_1}) holds for $z\in \partial \Omega_\varepsilon$ for all $\varepsilon\ll t$. 
 In other words, 
 \begin{equation}\label{eq:H_2}
 \varrho(z)\ge \frac{\sup_{\Omega\backslash  \Omega_t} (-\varrho)}{1-\sup_{\partial \Omega_t} \psi_t}\cdot (\psi_t(z)-1)
 \end{equation}
 holds for all $z\in \partial (\Omega_\varepsilon\backslash\overline{\Omega}_t)$. Since the RHS of (\ref{eq:H_2}) is psh, it follows from the\/ {\it maximal}\/ property of $\varrho$ that (\ref{eq:H_2}) (hence (\ref{eq:H_1})) holds for all $z\in \Omega_\varepsilon\backslash \Omega_t$ and $t\le r_0\ll1$ (so that $\overline{B}\subset \Omega_{r_0}$). Since $\varepsilon$ can be arbitrarily small, we conclude that (\ref{eq:H_1}) holds on $\Omega\backslash \Omega_t$. In particular, if $z\in   \Omega \backslash \Omega_{\alpha t}$ then  we have
 \begin{eqnarray}\label{eq:H_3}
 -\varrho(z) & \le & \sup_{\Omega\backslash  \Omega_t} (-\varrho) \cdot \frac{1-\inf_{\Omega\backslash \Omega_{\alpha t}}\psi_t}{1-\sup_{\partial \Omega_t} \psi_t}\nonumber\\
 & = & \sup_{\Omega\backslash  \Omega_t} (-\varrho) \cdot \left(1-\kappa_\alpha(t)\right).
   \end{eqnarray}
 
   Set $M(t):=\sup_{\Omega\backslash \Omega_t} (-\varrho)$. Since $-\log (1-x)\ge x$ for $x\in [0,1)$, it follows from (\ref{eq:H_3})  that 
      \begin{eqnarray*}
      \frac{\log M(t)}t-\frac{\log M(\alpha t)}t & \ge & \frac{\kappa_\alpha(t)}t .
               \end{eqnarray*}  
               Integration from $r$ to $r_0$ gives    
               \begin{eqnarray*}
  \int_{r}^{r_0} \frac{\kappa_\alpha(t)} t  dt  & \le &   \int_{r}^{r_0}   \frac{\log M(t)}t dt- \int_{r}^{r_0}  \frac{\log M(\alpha t)}t dt\\
               & = &  \int_{r}^{r_0}   \frac{\log M(t)}t dt- \int_{\alpha r}^{\alpha r_0}  \frac{\log M( t)}t dt\\
               & = &   \int_{\alpha r_0}^{r_0}   \frac{\log M(t)}t dt  -\int_{\alpha r}^{r}  \frac{\log M(t)}t dt\\
               &\le & (\log M(r_0)-\log M(\alpha r)) \log 1/\alpha,
                             \end{eqnarray*}
            because $M(t)$ is nondecreasing.  Thus (\ref{eq:H_0}) holds for $\varrho$ because $M(r_0)\le 1$.
            
            In general we may exhaust $\Omega$ by a sequence  $\{\Omega_{1/j}\}$ of hyperconvex domains. Let $\varrho_j$ be the relative extremal function of $\overline{B}$ relative to $\Omega_{1/j}$. Since $\inf_{\Omega\backslash \Omega_{\alpha t}} \psi_t\le \inf_{\Omega_{1/j}\backslash \Omega_{\alpha t}} \psi_t$, the previous argument yields 
            $$
          \sup_{\Omega_{1/j}\backslash \Omega_r} (-\varrho_j) \le \exp\left[-\frac{1}{\log 1/\alpha} \int_{r/\alpha}^{ r_0} \frac{\kappa_\alpha(t)} t dt  \right]  
            $$
            for all $j\ge j_r\gg 1$. Since $\varrho_j\downarrow \varrho$ (cf. \cite{Blocki}, Theorem 3.1.7), we get (\ref{eq:H_0}).
            \end{proof}  

    \begin{remark}
  {\rm Similar ideas were used by the author \cite{ChenCap} to estimate the Green function of planar domains in terms of capacity densities.}
\end{remark}  
            
            \section{Proof of Theorem \ref{th:Main}}
            We first prove the following 
                         
            \begin{proposition}\label{prop:HC}
             Let $\Omega$ be a bounded pseudoconvex domain in $\mathbb C^n$. Suppose there exist  numbers $\beta> 0$, $\gamma>1$, and a family $\{\Omega^t\}_{0<t\le t_0}$ of pseudoconvex domains such that $\overline{\Omega}\subset \Omega^t$ for all $t$ and     
            \begin{equation}\label{eq:M_0}
          \beta t^\gamma \le  \delta_{\Omega^t}(z):=d(z,\partial \Omega^t) \le  t,\ \ \ \forall\,z\in \partial \Omega.  
            \end{equation}
            For every $\tau<\frac1{\gamma-1}$ there exists a contant $C_\tau>0$ such that 
            \begin{equation}\label{eq:M_1}
            -\varrho(z) \le C_\tau (-\log \delta(z))^{-\tau}
            \end{equation}
            for all $z\in \Omega$ sufficiently close to $\partial \Omega$.
            \end{proposition}
            
           \begin{proof} 
         For $0<\varepsilon<1$ we define
            $$
            \psi_t :=\frac{\log 1/\delta_{\Omega^{\varepsilon t}}}{\log 2/(\beta \varepsilon^\gamma t^\gamma)}.
            $$ 
            Clearly, $\psi_t$ is a continuous  psh function on $\Omega$ satisfying $\sup_\Omega\psi_t<1$. For all $z\in \partial \Omega_t$ we have    
            $$
           t=\delta(z) \le  \delta_{\Omega^{\varepsilon t}}(z) \le t+\varepsilon t,
            $$             
           so that  
            $$
           \frac{\log 1/( t+\varepsilon t)}{\log 2/(\beta \varepsilon^\gamma t^\gamma)}\le \sup_{\partial \Omega_t} \psi_t \le \frac{\log 1/t}{\log 2/(\beta\varepsilon^\gamma t^\gamma)}.
            $$
            On the other hand, for $z\in \Omega\backslash \Omega_{\alpha t}$ we have
            $
            \delta_{\Omega^{\varepsilon t}}(z) \le \alpha t + \varepsilon t,
                       $
                       so that 
                       $$
                       \inf_{\Omega\backslash \Omega_{\alpha t}}\psi_t \ge \frac{\log 1/(\alpha t+ \varepsilon t)}{\log 2/(\beta\varepsilon^\gamma t^\gamma)}.
                       $$
                      Then we have
                       $$
                       \kappa_\alpha(t) \ge \frac{\log 1/(\alpha+\varepsilon)}{\log (2+2\varepsilon)/(\beta\varepsilon^\gamma t^{\gamma-1})}.
                       $$
            It follows from (\ref{eq:H_0}) that 
            \begin{eqnarray}\label{eq:M_2}
            -\varrho(z) & \le & \exp\left[-\frac{\log 1/(\alpha+\varepsilon)}{\log 1/\alpha} \int_{\delta(z)/\alpha}^{ r_0}\frac{dt}{t((\gamma-1)\log 1/t +\log (2+2\varepsilon)/(\beta\varepsilon^\gamma))}   \right]\nonumber\\
                        & \le & C_{\alpha,\beta,\varepsilon,\gamma} (-\log \delta(z))^{-\frac{\log 1/(\alpha+\varepsilon)}{(\gamma-1)\log 1/\alpha}}
            \end{eqnarray}  
            for all $z\in \Omega$ sufficiently close to $\partial \Omega$. In particular, $\Omega$ is hyperconvex. If we choose  $\varepsilon=\alpha$  sufficiently small, then the exponent in (\ref{eq:M_2}) can be arbitrarily close to $\frac1{\gamma-1}$.  Thus the proof is complete.
            \end{proof}

\begin{remark}
{\rm The same argument yields that if (\ref{eq:M_0}) holds for $\gamma=1$ then}
$$
-\varrho(z)\lesssim \delta(z)^\tau
$$
{\rm for some constant $\tau>0$.}
\end{remark}
            
            \begin{proof}[Proof of Theorem \ref{th:Main}]
            Since hyperconvexity is a local property (cf. \cite{KerzmanRosay}), it suffices to verify that for every $a\in \partial \Omega$ there exists a neighborhood $U$ of $a$ such that $\Omega\cap U$ is hyperconvex. After an affine transformation,  there exists a ball $B_r=\{z:|z|<r\}$ and a H\"older continuous function $g$ of order $\beta$ on 
            $$
            B'_r=\left\{(z',\text{Re\,}z_n)\in \mathbb C^{n-1}\times \mathbb R:|z'|^2+|\text{Re\,}z_n|^2<r^2\right\}
            $$
          where $z'=(z_1,\cdots,z_{n-1})$,    such that
            $$
          \Omega_{(r)}:=  \Omega\cap B_r=\left\{z\in B_r: \text{Im\,} z_n<g(z',\text{Re\,}z_n)\right\}.
            $$
           Fix $r_1<r$. We set
            $$
            \Omega_{(r_1)}^t:=\left\{z\in B_{r_1+t}: \text{Im\,} z_n<g(z',\text{Re\,}z_n)+t\right\}
            $$
            for $0<t\le t_0\ll1$. Clearly, these domains are pseudoconvex and contain $\overline{\Omega_{(r_1)}}$; moreover,
                        $$
            \delta_{\Omega_{(r_1)}^t}(z)\le t,\ \ \ \forall\,z\in \partial \Omega_{(r_1)}. 
            $$
            On the other hand, for given $z\in \partial{\Omega_{(r_1)}}$ we choose $z_t\in \partial \Omega_{(r_1)}^t$ such that $\delta_{\Omega_{(r_1)}^t}(z)=|z-z_t|$. If $z_t\in \partial B_{r_1+t}$, then
            $
                        \delta_{\Omega_{(r_1)}^t}(z) \ge t.
                                    $
                                    Thus we may assume without loss of generality that  
                                    $$
                  \text{Im\,} z_{n} \le g(z',\text{Re\,}z_{n})  \ \ \text{and}\ \                \text{Im\,} z_{t,n} = g(z'_t,\text{Re\,}z_{t,n})+t.
                                    $$ 
                                    If $\text{Im\,}z_{t,n}-\text{Im\,}z_n\ge t/2$, then $ \delta_{\Omega_{(r_1)}^t}(z)\ge t/2$. If  $\text{Im\,}z_{t,n}-\text{Im\,}z_n < t/2$, then
                                    $$
                                    g(z',\text{Re\,}z_{n})-g(z'_t,\text{Re\,}z_{t,n}) \ge  t-(\text{Im\,}z_{t,n}-\text{Im\,}z_n)> t/2,
                                    $$
                                    while 
                                    \begin{eqnarray*}
                                    g(z',\text{Re\,}z_{n})-g(z'_t,\text{Re\,}z_{t,n})  & \le & c|(z'_t,\text{Re\,}z_{t,n})-(z',\text{Re\,}z_{n})|^\beta   \\
                                    & \le &   c|z_t-z|^\beta   = c   \delta_{\Omega_{(r_1)}^t}(z)^\beta.                          
                                    \end{eqnarray*}
                                    Thus we have
                                    $$
                                     \delta_{\Omega_{(r_1)}^t}(z)\ge (2c)^{-1/\beta}  \cdot t^{1/\beta}.
                                                                       $$
                     By Proposition \ref{prop:HC}, we conclude that for every $\tau<\frac{\beta}{1-\beta}$ there exists  a negative continuous psh function $\rho$ on $\Omega_{(r_1)}$ satisfying
\begin{equation}\label{eq:M_4}
-\rho(z) \lesssim (-\log \delta(z))^{-\tau}
\end{equation}
 for all $z\in \Omega$ sufficiently close to $\partial \Omega_{(r_1)}$.            
            \end{proof}

\begin{remark}
{\rm (1) In view of \cite{ChenH-index}, it is natural to ask whether the local estimate (\ref{eq:M_4})  yields certain  estimate for the Bergman distance. We will come back to this problem in a future paper}.

{\rm (2) It is expected that $-\rho(z) \lesssim (-\log \delta(z))^{-\tau}$ holds also globally on bounded pseudoconvex domains with H\"older boundary in $\mathbb C^n$.}
\end{remark}
            
            \section{The Lipschitz case}
            
            \begin{proposition}\label{prop:Lip}
           Let $\Omega$ be a bounded pseudoconvex domain with Lipschitz boundary. Then there exists a number $\tau>0$ such that
           $$
           -\varrho(z)\lesssim \delta(z)^\tau.
           $$
            \end{proposition}
            
            \begin{proof}
           By  \cite{Demailly87}, there exist a number $c>0$ and a family $\{v_t\}_{0<t\le t_0}$ of continuous psh functions on $\Omega$ satisfying
           $$
           \log 1/(\delta(z)+t)-c<v_t(z)<\log 1/(\delta(z)+t).
           $$ 
           Given $0<\varepsilon<1$. We set 
           $$
            \psi_t :=\frac{v_{\varepsilon t}}{\log 2/(\varepsilon t)}.
            $$ 
          Clearly, $\psi_t$ is psh on $\Omega$ and satisfies $\sup_\Omega \psi_t<1$.  For all $z\in \partial \Omega_t$ we have    
            $$
           \frac{\log 1/(t+\varepsilon t)-c}{\log 2/(\varepsilon t)}\le \sup_{\partial \Omega_t} \psi_t \le \frac{\log 1/(t+\varepsilon t )}{\log 2/(\varepsilon t)}.
            $$
            On the other hand, for $z\in \Omega\backslash \Omega_{\alpha t}$ we have
                       $$
                       \inf_{\Omega\backslash \Omega_{\alpha t}}\psi_t \ge \frac{\log 1/(\alpha t+\varepsilon t)-c}{\log 2/(\varepsilon t)}.
                       $$
                        Then we have
                       $$
                       \kappa_\alpha(t) \ge \frac{\log (1+\varepsilon)/(\alpha+\varepsilon)-c}{\log (2+2\varepsilon)/\varepsilon+ c}.
                       $$
                       We choose $\alpha,\varepsilon$ sufficiently small so that $\log (1+\varepsilon)/(\alpha+\varepsilon)>c$. 
            It follows from (\ref{eq:H_0}) that   
            $$
            -\varrho (z) \lesssim \delta(z)^\tau
            $$        
            where 
            $$
            \tau= \left(\log 1/\alpha\right)^{-1} \cdot \frac{\log (1+\varepsilon)/(\alpha+\varepsilon)-c}{\log (2+2\varepsilon)/\varepsilon+ c}.
                        $$
   \end{proof}

 \begin{remark}
 {\rm Although Proposition \ref{prop:Lip} is weaker than Harrington's result, it is sufficient for many purposes (compare \cite{ChenH-index}).}
\end{remark}

\section{On domains beyond H\"older boundary regularity}
   
 \begin{proposition}\label{prop:BeyondHC}
             Let $\Omega$ be a bounded pseudoconvex domain in $\mathbb C^n$. Suppose there exists   a family $\{\Omega^t\}_{0<t\le t_0}$ of pseudoconvex domains such that $\overline{\Omega}\subset \Omega^t$ for all $t$ and     
            $$
        \eta(t)\le  \delta_{\Omega^t}(z)\le  t,\ \ \ \forall\,z\in \partial \Omega
            $$ 
where $\eta$ is a continuous increasing function satisfying 
\begin{equation}\label{eq:BeyondHC}
\int_0^{r_0} \frac{dt}{t\log [t/\eta(t)]}=\infty
\end{equation}
for some $r_0\ll 1$. Then $\Omega$ is hyperconvex.
            \end{proposition}
            
           \begin{proof}           
For fixed $0<\varepsilon <1$ the function
            $$
            \psi_t :=\frac{\log 1/\delta_{\Omega^{\varepsilon t}}}{\log 2/\eta(\varepsilon t)}
            $$ 
            is psh  on $\Omega$ and satisfies $\sup_\Omega\psi_t<1$. Since    
            $$
           t\le  \delta_{\Omega^{\varepsilon t}}(z) \le t +\varepsilon t,\ \ \ \forall\,z\in \partial \Omega_t,
            $$             
           we have
            $$
           \frac{\log 1/(t+\varepsilon t)}{\log 2/\eta(\varepsilon t)}\le \sup_{\partial \Omega_t} \psi_t \le \frac{\log 1/t}{\log 2/\eta(\varepsilon t)}.
            $$
            On the other hand, for $z\in \Omega\backslash \Omega_{\alpha t}$ we have
            $
            \delta_{\Omega^{\varepsilon t}}(z) \le \alpha t + \varepsilon t,
                       $
                       so that 
                       $$
                       \inf_{\Omega\backslash \Omega_{\alpha t}}\psi_t \ge \frac{\log 1/(\alpha t+\varepsilon t)}{\log 2/\eta(\varepsilon t)}.
                       $$
                       Fix $\alpha,\varepsilon$ with $\alpha+\varepsilon<1$. Then we have
                       $$
                       \kappa_\alpha(t) \ge \frac{\log 1/(\alpha+\varepsilon)}{\log [(2+2\varepsilon)t/\eta(\varepsilon t)]}.
                       $$
            It follows from (\ref{eq:H_0}) that   for suitable positive constant $\tau$ depending only on $\alpha,\varepsilon$,
            \begin{eqnarray*}
            -\varrho(z) & \le & \exp\left(-\tau \int_{\delta(z)/\alpha}^{ r_0}\frac{dt}{t\log [t/\eta(\varepsilon t)]}  \right)\\
                        & \rightarrow & 0\ \ \ (\text{as}\ z\rightarrow \partial \Omega). 
            \end{eqnarray*}  
           Thus $\Omega$ is hyperconvex.    
            \end{proof}   

  Condition (\ref{eq:BeyondHC}) is satisfied for instance, if $\eta(t)=t (-\log t)^{\log t}$, much faster than any power $t^\gamma$, $\gamma>0$.

\end{document}